\documentclass[10pt]{article}
\usepackage{latexsym,amssymb}
\usepackage{enumerate, verbatim, mathrsfs}
\usepackage[stable]{footmisc}
\usepackage{graphics}

\usepackage{amsthm}
\usepackage{amssymb}
\usepackage{latexsym}
\usepackage{longtable}
\usepackage{epsfig}
\usepackage{amsmath}
\usepackage{hhline}
\usepackage{enumerate}
\usepackage{pstricks, graphicx, multido, pst-node}
\usepackage{amssymb, amsmath, amsfonts, amsthm, latexsym, enumerate, verbatim}
\usepackage{epsfig}
\usepackage{pst-grad} 
\usepackage{pst-plot} 

\newtheorem{athm}{\sc Theorem}
\newtheorem{thm}{\sc Theorem}[section]
\newtheorem{prop}[thm]{\sc Proposition}
\newtheorem{lem}[thm]{\sc Lemma}
\newtheorem{cor}[thm]{\sc Corollary}

\newcommand{\abs}[1]{\ensuremath{\left\vert #1 \right\vert}}
\newcommand{\qfrac}[2]{\ensuremath{\left[\,{#1 \atop #2}\,\right]}}

\newcommand{\nc}[1]{\ensuremath{\abs{#1}}}

\newcommand{\alt}[1]{\ensuremath{\rm{Alt}({\ensuremath{#1}})}}
\newcommand{\sym}[1]{\ensuremath{\rm{Sym}({\ensuremath{#1}})}}

\newcommand{\gG}{\Gamma}

\newcommand{\cC}{\mathbb{C}}
  
\newcommand{\na}{\mathbb{N}}  
  
\newcommand{\za}{\mathbb{Z}} 

\newcommand{\di}{{\rm d}}

\newcommand{\ad}{{\rm ad}}
\newcommand{\ins}{{\rm ins}}
\newcommand{\del}{{\rm del}}
\newcommand{\dne}{\hfill $\Box$ \vspace{0.3cm}}
\newcommand{\pf}{{\it Proof:\,\,\,}}

\begin{document}

\title{Metric intersection problems in \\ 
Cayley graphs and the Stirling recursion}

\author{  Teeraphong Phongpattanacharoen \\\\ 
{\small\sc Department of Mathematics and Computer Science}\\
{\small\sc Chulalongkorn University, Bangkok 10330, Thailand}  \\\\
{\small \it \,and\,\,} Johannes Siemons\\
\\ {\small\sc School of Mathematics}\\
{\small\sc University of East Anglia, Norwich NR4 7TJ, UK}
}

\date{\scriptsize Version of 17 February,  2012; typeset \today}

\maketitle

\begin{abstract}
\noindent In \sym{n} with $n\geq 5$ let $H$ be a conjugacy class of elements of order $2$ and let  $\gG$ be  the Cayley graph whose vertex set is the group $G$ generated by $H$ (so $G=\sym{n}\text{\, or \,}\alt{n}$) and whose edge set is determined by $H.$
We are interested in the metric structure of this graph. In particular, for $g\in G$ let $B_{r}(g)$ be the metric ball in $\gG$ of radius $r$ and centre $g.$ We show that the intersection numbers $ \Phi(\gG;\;\,r,g):=|\,B_{r}(e)\,\cap\,B_{r}(g)\,|$ are generalized Stirling functions in $n$ and $r.$ The results are motivated by the study of error graphs in Levenshtein~\cite{lv1,lv2,lv3} and related reconstruction problems.



\end{abstract}

\section{Metric Intersections in Cayley graphs}

Let $G$ be a finite group and let $H$ be a subset of $G$ which generates it. We assume that $H=\{ h^{-1}\,:\,h\in H\}$ and that $H$ does not contain the identity element $e$ of $G.$ In this situation $H$ defines an undirected simple Cayley graph $\gG=\gG_{G}^{H}$  on the  vertex set $G.$ The usual graph distance on $\gG$ is denoted ${\rm d}\!:G\times G\to \na\cup \{0\};$ thus  $d={\rm d}(u,w)$  for $u,\,w\in G$  is the least number of $h_{i}\in H$ so that $w=uh_{1}...h_{d}.$ This defines a metric on $\gG.$ 
For an integer $r\geq 0$ and $g\in G$ the set $$B_{r}(\gG,g):=\big\{ \,u\in G\,\,:\,\,\di(g,u)\leq r\,\big\}$$ 
is the {\it metric ball of radius $r$ and centre}\, $g.$ 
In this paper we are interested in the {\it  metric intersection numbers} $$ \Phi(\gG;\,r,g):=|\,B_{r}(\gG,e)\,\cap\,B_{r}(\gG,g)\,|$$  
considered as a function on $G$ for  $r\geq 0.$ Notice in particular, $ \Phi(\gG;\,r,e)$ is the cardinality of $B_{r}(\gG,e),$ and this information also provides the diameter of $\gG$.  We determine these intersection numbers when $H$ is a conjugacy class of elements of order two in the symmetric group ${\rm Sym}(n)$ for $n\geq 5.$ In this case $G$ is $\sym{n}$ or $\alt{n},$  depending on $H.$  

\medskip
The interest in  intersection numbers of this type stems from the error graphs discussed in Levenshtein~\cite{lv2} in the context of combinatorial reconstruction and error correction problems. In this situation  $\Phi(\gG;r,g)$ is a measure for the amount of information needed to reconstruct data distorted by transpositional errors, for more detail see Section~7\, of this paper  and~\cite{js}. In the same paper we determined the maximum $N(\gG,r):=\max\{\Phi(\gG;r,g)\,:\,e\neq g\in G\,\}$ for all $r\geq 1$ when $\gG$ is the {\it transposition Cayley graph}\, on ${\rm Sym}(n);$ its generating set consists of all  transpositions on $\{1..n\}.$ This Cayley graph is well-known from many  applications in computer science.  The results of the current paper are a key for determining  $N(\gG,r)$ in  Cayley graphs on \sym{n} generated by an arbitrary conjugacy class of elements of order~$2.$ 

Metric intersection numbers of this kind occur  in various contexts, from distance regular graphs and distance statistics on symmetric groups~\cite{dia} to generation problems in finite groups.  For instance, if $\gG$ is a distance regular graph then $\Phi(\gG;r,g)$ is determined completely by the usual graph intersection numbers, see~\cite{BCN}. It may therefore be of interest to determine  intersection numbers for graphs which are not distance regular, as is the case here,  but which still have a high degree of symmetry. In the context of group generation  $\Phi(\gG;r,e)$  is the growth rate studied in geometric group theory and metric geometry, see also~\cite{tao}.

\medskip
In this paper we emphasize the connections to combinatorial enumeration.  We show that  for all Cayley graphs considered here $\Phi(\gG;r,g)$ is a {\it  Stirling function,} in the following sense.   Let $t\geq 0$ be an integer and let $\na_{t}=\{n\in \za\,\,:\,\,n\geq t\}.$ Then a function $\na_{t}\times\mathbb{Z}\to\na_{0}$ given by $(n,m)\mapsto\big[{n\atop m}\big]$ is a {\it Stirling function with threshold $t$} provided it satisfies the Stirling recurrence $$(S):\quad \Big[{n\atop m}\Big]=\Big[{n-1\atop m-1}\Big]+(n-1)\Big[{n-1\atop m}\Big]\,\text{,\,\,and\,\,\,$\Big[{n\atop m}\Big]=0$ for all $m>n,$}\quad$$ for all $n>t$ and  $m\in \mathbb{Z}.$ The usual {\it unsigned Stirling numbers of the first kind} 
are a Stirling function with threshold $t=1.$

\medskip
To state the main results we need some definitions. For $k$ and $n\geq 2k$ 
a permutation in $\sym{n}$ of cycle shape $(a_{1}b_{1})...(a_{k}b_{k})$  is called a  {\it $k$-transposition}.  Let  $H$ be the set of all $k$-transpositions.   Then $H$ generates  $\sym{n}$ if $k$ is odd, or $\alt{n}$ if $k$ is even and $n\neq 4.$ The corresponding Cayley graph is the {\it $k$-transposition Cayley graph\,} on the symmetric or alternating group respectively. It is denoted by $\gG^{k}_{n}.$  If  $g$ is a permutation on some set then its {\it support size} ${\rm supp}(g)$ is the number of  elements that are moved by $g.$ If $n$ is any integer $\geq {\rm supp}(g)$ then we may identify the set  with a subset of $\{1..n\}$  and regard $g$ as an element of $\sym{n}.$

\bigskip
\begin{athm}\label{main1} Let  $k$ be an even integer or $k=1,$ and let $g$ be an even permutation  with support size $s.$  
Then there is a Stirling function $[{n\atop m}]_{g}$ with threshold $t=\max\{s, 4k\}$ so that 
$$ \Phi(\gG^{k}_{n};\;\,r,g)\,\,=\,\,\big[{n\atop \,\,n-rk\,\,}\big]_{g}$$ for all $n>t$ and  $r\geq 3$ or $r\geq 0$ when $k=1.$ \end{athm}

This is proved in Theorem~\ref{Gamma1}\, (Section~\ref{sect32})\, and Theorem~\ref{Gamma8}\, (Section 6)\, where we give additional  information about these Stirling functions. Here we emphasize that the function $[{n\atop m}]_{g}$ depends on $g$ only via its initial values at $n=t$ but not on $k.$ The same applies to the next theorem.

\bigskip
\begin{athm}\label{main2} Let $k\geq 3$ be odd and let $g$ be a permutation with support size  $s.$ 
Then there is a Stirling function $[{n\atop m}]_{g}$ with threshold $t=\max\{s,4k\}$ so that 
$$ \Phi(\gG^{k}_{n};\;\,r,g)\,\,=\,\,\big[\begin{array}{c}n\\n-rk\end{array}\big]_{g}\,\,+\,\,\, \big[\begin{array}{c}n\\n-(r-1)k\end{array}\big]_{g}$$ for all $n>t$ and $r\geq 3.$
\end{athm}

This is proved in Section 6. Taken together the main theorems show that 
the intersection numbers for the Cayley graphs corresponding to a single conjugacy class of involutions (elements of order $2)$  are governed by the same principle. As a by-product, they also provide the diameter of the graphs. For many applications and asymptotic considerations it is sufficient to work with the Stirling recurrence irrespective of initial values. However, if the exact values for the intersection numbers are required then one needs to determine  the initial values of $\Phi(\gG^{k}_{n};\;\,r,g)$ for $n=t$ explicitly by direct computation. These initial values are easily obtained for small values of $t.$

\bigskip
In Section 2\, we collect prerequisites. In Section 3\, we review the ordinary transposition Cayley graph and prove a Deletion Lemma which is essential for  induction. In Section 4\, the intersection numbers for the $2$-transposition graph are determined. The main work for $k$-transposition graphs in general starts in Section 5 with an explicit description of the metric balls $B_{r}$ in the $k$-transposition graphs $\gG^{k}_{n}.$ As we have mentioned before, these questions arose from error correction and reconstruction problems in combinatorics. In Section 7 we give an account of the connection to reconstruction numbers and mention some of the  literature relevant to this aspect of intersection numbers. We conclude with suggestions for open problems.

\section{Preliminaries}\label{Sect2}

In this section we collect some prerequisites.  Permutations are expressed in cycle notation and our convention is that we  write permutations on the right of the argument.  So the product $(1,2,3)\cdot(2,3)$ is $(1,3)(2).$ Other details on permutations can be found in \cite{Dixon} or \cite{Sagan}.

\subsection{The Stirling Recursion}

Let $t\geq 0$ be an integer and let $\na_{t}=\{n\in \za\,\,:\,\,n\geq t\}.$
We call the function $f\!:\na_{t}\times\mathbb{Z}\to\na_{0}$ a {\it (generalized) Stirling function}  if  $f$ satisfies the  \emph{Stirling recursion} 
\begin{align}
f(n,k) = f(n-1,k-1) + (n-1)f(n-1,k)\,,\,\,\,\text{and}\,\, f(n,k)=0\quad\text{when $k>n,$}\label{STRNo}
\end{align}
for all $n>t$ and all $k\in \mathbb{Z}.$
The integer $t$ is the {\it threshold} of  $f$ and the values of $f$ on $\{t\}\times \za$ are the {\it initial values} of $f.$ Evidently a Stirling function is determined completely by its threshold and initial values.

The ordinary unsigned Stirling numbers of the first kind, often denoted $c(n,k),$
are a Stirling function with threshold $t=1$ and initial values $c(1,k)=0$ for all $k\neq 1$ and $c(1,1)=1.$ Donald Knuth introduced the notation 
$c(n,k)=[{n\atop k}]$ to remind us that this function counts the number of choices of permutations of $\{1..n\}$ with exactly $k$ cycles, see~\cite{AB}. This notation is  preferable also here, the recurrence \,(\ref{STRNo})\, then is Equation~$(S)$\, in the Introduction.

 Substituting $n-r:=k$ we have the equivalent recurrence 
 \begin{align}
\Big[{n\atop n-r}\Big]=\Big[{n-1\atop n-r-1}\Big]+(n-1)\Big[{n-1\atop n-r}\Big]
\,,\,\,\text{and}\,\, \Big[{n\atop n-r}\Big]=0\quad\text{when $r<0,$}
\label{STRNoo}
\end{align}
for all $n>t$ and all $r\in \mathbb{Z}.$ This form  is more useful for us later. Further information about $r$-Stirling numbers and generalized Stirling numbers - all of these are Stirling numbers in the sense of the definition here - can be found in the papers by Bickel~\cite{TB},\, Broder~\cite{AB, AB2},\, Hsu~\cite{hsu} and  Mes\H{o}~\cite{mz}.



\subsection{Metric Balls and Spheres in Cayley Graphs}\label{sect22}

We introduce some notation for Cayley graphs. Let $G$ be a group with identity denoted by $e.$ Suppose that $G$ is generated by a subset $H$ of $G$  so that \vspace{-2mm}\begin{enumerate}[\quad(i)]\item $e\not\in H,$\vspace{-2mm}\item $H=\{h^{-1}\,:\,h\in H\}$ and \vspace{-2mm}\item$H$ is a union of conjugacy classes of $G.$\end{enumerate}  \vspace{-2mm} We form the Cayley graph  on $G$ with generating set $H$ by joining two vertices $u\neq v$ by  an edge  if and only if there exists some $h\in H$ with $uh=v.$ This defines an undirected  graph  without multiple edges or loops on the vertex set $G.$ It will be denoted by $\gG^{H}_{G}.$ For two vertices $u$ and $v$ the usual graph distance ${\rm d}(u,v)$ is the least number $k$ so that  there are $h_{1},h_{2},..,h_{k}\in H$ with $uh_{1}h_{2}\cdots h_{k}=v.$ Such elements always exist as $H$ generates $G.$

For the integer $r\geq 0$ and $g\in G$ define the {\it metric sphere} and 
{\it metric ball of radius} $r$ and {\it centre} $g$ in $\gG=\gG^{H}_{G}$ as $$S_{r}(\gG,g):=\{ u\in G\,: {\rm d}(u,g)=r\}\text{\quad and \quad}   
B_{r}(\gG,g):=\{ u\in G\,: {\rm d}(u,g)\leq r\}.$$ respectively. 
Abbreviate $S_{r}=S_{r}(\gG)=S_{r}(\gG,e)$ and 
$B_{r}=B_{r}(\gG):=B_{r}(\gG,e)$ where the context is clear. When $X$ is a subset of $G$ and $g$ is an element of $G$ we denote by $Xg$ the set $\{xg\,:\,x\in X\},$ and $gX$ has a similar meaning.  

\bigskip
\begin{prop} \label{Cayley22} Let $G$ be generated by a set $H$ satisfying (i), (ii) and (iii). Let $\gG:=\gG^{H}_{G}$ be the corresponding Cayley graph, and let $g\in G.$ Then \vspace{-4mm}
\begin{enumerate}[(i)]\item the left- and right-multiplication maps $\lambda_{g}:G\to G$ and $\rho_{g}:G\to G$ given by  $\lambda_{g}(u)=gu$ and $\rho_{g}(u)=ug,$ respectively, are automorphisms of $\gG,$ and  \vspace{-2mm}
\item for all $r\geq 0$ we have $gS_{r}(\gG)=S_{r}(\gG,g)=S_{r}(\gG) g$ and 
$g B_{r}(\gG)=B_{r}(\gG,g)= B_{r}(\gG)g.$
 \vspace{-2mm}\end{enumerate}\end{prop}

\pf If $\{u,\,uh\}$ with $h\in H$ is an edge then $\{gu,\,guh\}$ and $\{ug,\,ug(g^{-1}hg)\}$ are edges, as $g^{-1}hg\in H.$ The second property follows from (i).\dne

\medskip
We define  the {\it metric intersection numbers}\, for $\gG=\gG^{H}_{G}$ by $$ \Phi(\gG;\,r,g):=|\,B_{r}(\gG)\,\cap\,B_{r}(\gG,g)\,|\,.$$  
For $x\in G$  we have $xB_{r}=B_{r}x$ by Proposition~\ref{Cayley22}\, and hence $ \Phi(\gG;\,r,g)=|B_r \cap B_rg|=|x^{-1}(B_r \cap B_rg)x|=|x^{-1}B_rx \cap x^{-1}B_{r}x(x^{-1}gx)|=|B_r \cap B_r(x^{-1}gx)|= \Phi(\gG;\,r,x^{-1}gx).$ Furthermore, $(B_r \cap B_rg)g^{-1}=B_rg^{-1} \cap B_r.$ Therefore

\medskip 
\begin{lem}\label{class} Let $G$ be generated by a set $H$ satisfying (i), (ii) and (iii). Then for fixed $r$  the function $g\mapsto  \Phi(\gG;\,r,g)$ is constant on the conjugacy classes of $G$ and $\Phi(\gG;\,r,g)=\Phi(\gG;\,r,g^{-1})$ for all $g\in G.$
\end{lem}

{\sc Comments:}\, (i) As a function of $r$ one can view $\Phi(\gG;\,r,e)$ as a  measure for the speed by which $H$ generates $G.$ This is a standard notion for the {\it growth rate}\, in finitely generated groups. For fixed $r$ on the other hand 
$g\mapsto  \Phi(\gG;\,r,g)$ is a measure for the {\it movement } of $B_{r}(\gG)$ by  group elements, or for a  covering property in the group. 

(ii) It follows from the lemma that  $\Phi(\gG;\,r,g)$ as a function on $G$ is a linear combination of ordinary complex characters of $G.$ We have made some experiments in the symmetric groups. The coefficients expressing $\Phi(\gG;\,r,g)$  in terms of characters do not have any straightforward interpretation. Clearly, this is a natural question independently of Cayley graphs. 
We conjecture that $$\{\,\Phi(\gG^{H}_{G};\,1,g)\,:\, H\text{\,\, is a union of conjugacy classes of $G$}\}$$  spans the space of class functions on $G=\sym{n}$ or $\alt{n}$ respectively. We could not find much literature on this  problem.

\bigskip
\subsection{ The Insertion and Deletion Map on the Symmetric Group}\label{VAR}
Let $G_{n}$ denote the symmetric group  on the set $\{1..n\}.$ We express permutations as disjoint cycles written on the right of the argument. 
For instance, $(1,2)(3)\cdot (1,2,3)=(1,3)(2).$ If $g$ has $n_{t}$ cycles of length $t$ for $t=1..n$ we say that $g$ has {\it cycle type} $1^{n_{1}}2^{n_{2}}...n^{n_{n}}$ and  the total number of cycles is denoted by $|g|=\sum_{t=1..n}\, n_{t}.$ When counting cycles it is essential to include cycles of length $1.$ The {\it support} of $g$ is the set  ${\rm Supp}(g):=\{a\in\{1..n\}\,\,:\,\,a^{g}\neq a \}$ moved by $g$ and the {\it support size} \,is ${\rm supp}(g):=|{\rm Supp}(g)|.$ In particular, ${\rm supp}(g)=0$ if and only if $g$ is the identity permutation.

Let $k\geq 1$ be an integer with $n\geq 2k.$ Then  a {\it $k$-transposition} of $\{1..n\}$ is a permutation of cycle shape $(a_{1}b_{1})...(a_{k}b_{k}),$ consisting of $k$ two-cycles and $n-2k$ fixed points. In particular, $1$-transpositions are the usual transpositions on $\{1..n\}.$ Observe the important  multiplication rule for transpositions:\, if $(a_{1},a_{2},.,a_{r})$ and $(b_{1},b_{2},.,b_{s})$ are disjoint cycles then \begin{align}(M)\!:\,\,(a_{1},a_{2},.,a_{r})(b_{1},b_{2},.,b_{s})\,\cdot (a_{1},b_{1})=(a_{1},a_{2},.,a_{r},b_{1},b_{2},.,b_{s}).\label{MM}\end{align}
It says that multiplying by a transposition either glues together two cycles, as in \,(M),\, or cuts a cycle into two. The latter can be seen by  multiplying this equation again by $(a_{1},b_{1}).$

For $0\leq j\leq n$ define the {\it insertion map} $\ins_{j}\!:\,G_{n}\to G_{n\!+\!1}$ as follows. For $0<j$ and $g\in G_{n}$ let $\ins_{j}(g)$ be obtained by inserting $n\!+\!1$ after $j$ in the same cycle as $j.$ Thus 
\begin{align}\ins_{j}(g)\!: \left\{\begin{array}{ccc}a\neq j,\, n\!+\!1& \mapsto & a^g \medskip\\j & \mapsto & n\!+\!1\medskip\\n\!+\!1 & \mapsto & j^g\,.\end{array}\right.\label{aINSJ}\end{align}
For $j=0$ let $\ins_{0}(g)$ be obtained  by appending  the one-cycle $(n\!+\!1)$ to $g,$ thus  \begin{align}\ins_{0}(g)\!: \left\{\begin{array}{ccc}a\neq n\!+\!1& \mapsto & a^g \medskip\\n\!+\!1 & \mapsto & n\!+\!1\,.\end{array}\right.\label{aINS0}\end{align}
For instance, $\ins_{2}\!:\,G_{4}\to G_{5}$ maps $(1)(2,4)(3)$ to $(1)(2,5,4)(3)$ and $(2)(1,4)(3)$ to $(2,5)(1,4)(3)$ while $\ins_{0}\!:\,G_{4}\to G_{5}$ sends 
$(1)(2,4)(3)$ to $(1)(2,4)(3)(5).$ 

The insertion map $\ins_{j}$ is injective for all $0\leq j\leq n$ and  $\ins_{i}(G_{n})\cap \ins_{j}(G_{n})=\emptyset$ if $i\neq j.$ In addition, for $0<j$ the number of cycles in a permutation is preserved by $\ins_{j}$ and increased by one by $\ins_{0}.$

\medskip
The {\it deletion map} $\del\!:\,G_{n\!+\!1}\to G_{n}$ is defines as follows. If a permutation in $G_{n+1}$  is expressed in cycle notation then $\del$  removes the letter $n\!+\!1$ from the cycle  containing it. Evidently,  if $0\leq j\leq n$ then $\del\!\circ\!\ins_{j}\!:\,G_{n}\to G_{n}$ is the identity map and so $\del$ restricted to the set $\ins_{j}(G_{n})$ is a bijection. On  $\ins_{0}(G_{n})$ the map decreases the number of cycles by one while it preserves this number on   $\ins_{j}(G_{n})$ when $1\leq j\leq n.$ The following can be verified easily  by evaluating \,(\ref{aINSJ})\, and \,(\ref{aINS0}):

\medskip 
\begin{lem}\label{Ins} Let $u,\,v$ be permutations of $\{1..n\}.$  Then \begin{align}
\ins_j(uv)=\ins_{j}(u)\cdot\ins_0(v) \label{MER}
\end{align} for all $j\in \{0..n\}.$
\end{lem}

If we write $g={\rm id}\cdot g$ then the lemma gives \begin{align}\ins_{j}(g)=\ins_{j}(\rm id)\cdot\ins_{0}(g)=(j,n\!+\!1)\cdot g\,.\label{AUTINSJ}\end{align} In particular, for $0<j$ the insertion map $\ins_{j}$  can be realized as multiplication by  the transposition $(j,n\!+\!1)$ on the left. Hence $\ins_{j}(G_{n})=(j,n\!+\!1)G_{n}$ and $\del\!:(j,n\!+\!1)G_{n}\to G_{n}$ with $\del((j,n\!+\!1)g)=g$ is a bijection. 
It is natural to identity $\ins_{0}(G_{n})$ with $G_{n}$ by setting  $\ins_{0}(g)=g.$

\bigskip
Now suppose that $H_{n}$ is a union  of $G_{n}$-conjugacy classes that generates $G_{n}.$ (It is sufficient that $H_{n}$ contains an odd permutation, since $\alt{n}$ is simple for $n\geq 5.$)  Denote the corresponding Cayley graph by $\gG_{n}:=\gG^{H_{n}}_{G_{n}}.$  Let $H_{n\!+\!1}=\{h^{g}\,:\,h\in H,\,g\in G_{n\!+\!1}\}. $ Then $H_{n\!+\!1}$ generates $G_{n\!+\!1}$  and $H_{n}=H_{n\!+\!1}\cap G_{n}.$ We denote the corresponding Cayley graph by $\gG_{n\!+\!1}:=\gG^{H_{n\!+\!1}}_{G_{n\!+\!1}}.$

If we identify  $G_{n}=\ins_{0}(G_{n})$ then $\gG_{n}$ becomes an induced subgraph of
$\gG_{n\!+\!1}$ since $u^{-1}v\in H_{n\!+\!1}$ with $u,\,v\in G_{n}$ implies  $u^{-1}v\in H_n.$ Thus $\ins_{0}:\gG_{n}\to \gG_{n\!+\!1}$ is an embedding of graphs. Similarly, using~~(\ref{AUTINSJ})\, for $j>0,$ we may factorize  $\ins_{j}:G_{n}\to G_{n\!+\!1}$ as $\ins_{j}=\lambda_{(j,n\!+\!1)}\circ\ins_{0}:G_{n}\to \ins_{0}G_{n}\to G_{n\!+\!1}$ where $\lambda_{(j,n\!+\!1)}$ is left-multiplication by the transposition $(j,n\!+\!1).$ By Proposition~\ref{Cayley22}\, this map  is a graph isomorphism. Hence also $\ins_{j}:\gG_{n}\to \gG_{n\!+\!1}$ is an embedding of graphs.

\medskip
\begin{prop}\label{Embed22} Let $\gG_{n}$ and $\gG_{n\!+\!1}$ be given as above. Then for all $j$ with $0\leq j\leq n$ the maps $\ins_{j}:G_{n}\to G_{n\!+\!1}$ induce graph embeddings $\ins_{j}:\gG_{n}\to \gG_{n\!+\!1}.$ 
\end{prop}

The embedding $\ins_{j}:\gG_{n}\to \gG_{n\!+\!1}$ is not always isometric. Examples where ${\rm d}_{\gG_{n}}(u,v)$ is strictly larger than ${\rm d}_{\gG_{n+1}}(\ins_{0}(u),\ins_{0}(v))$ can be found in Section 5.1. 

\bigskip
\section[title]{The Transposition Cayley Graph}

Let $G_{n}$ be the symmetric group on $\{1..n\}$ with $n\geq 2$ and let $H$ be the set of all transpositions in $G_{n}.$ Then $H$ generates $G_{n}$ and the resulting Cayley graph $\gG_n^{1}:=\gG^{H}_{G_{n}}$ is  the {\it transposition Cayley graph } on  $G_{n}.$ If the context is clear we will write simply  $\gG_{n}.$
 As before ${\rm d}(x,\,y)$ is the graph distance in $\gG_{n}.$ For $r\geq 0$ we set $S_r:=S_r(\gG_n,e)$ and 
 $B_r:=B_r(\gG_n,e).$ If $g\in G_{n}$ we put $S_{r}(g):=S_r(\gG_n,g)$ and 
 $B_r(g):=B_r(\gG_n,g).$ Recall that $gS_{r}=S_r(g)=S_{r}g$ and 
 $gB_{r}=B_r(g)=B_{r}g$ by Proposition~\ref{Cayley22}. Our task is to determine the intersection numbers 
$$\Phi(\gG_{n};\,r,g)=\abs{B_r \cap B_rg}$$ for all $n\geq 1,$ $r\leq n-1$ and $g\in G_{n}.$

First we determine  the ball $B_{r}$ of radius $r$ in $\gG.$ This is well-known but the  proof illustrates the connection to the Stirling recursion and the methods to be used later. Recall that if $g$ is a permutation on $\{1..n\}$ then $|g|$ denotes the number of cycles of $g.$

\medskip
\begin{thm} \label{BB1}For  $n\geq 2$ and $r\geq 0$ let $S_{r}$ and $B_{r}$ denote the sphere and ball of radius $r$ respectively in the transposition Cayley graph $\gG^{1}_{n}.$  Then \newline
(i) \,  $S_{r}=\{ x\in G_{n}\,:\, |x|=n-r\}$ and $B_{r}=\{ x\in G_{n}\,:\, |x|\geq n-r\},$ and\newline
(ii) \,$[{n\atop n-r}]:= |B_r{}|$ satisfies the Stirling relation
$$\big[{n\atop n-r}\big]=\big[{n-1\atop n-1-r}\big]+(n-1)\big[{n-1\atop n-r}\big]\quad \text{and \quad $\Big[{n\atop m}\Big]=0$ for  $m>n.$}$$

\vspace{-2mm}
for all $n\geq 3$ with initial values  $[{2\atop 2}]=1$ and \, $[{2\atop 2-r}]=2$ for $r\geq 1.$ 
\end{thm}

Note for instance, the initial values $|B_{0}(\gG_{2})|<|B_{1}(\gG_{2})|=|\sym{2}|$ and the recursion imply that $|B_{n-2}(\gG_{n})|<|B_{n-1}(\gG_{n})|=|\sym{n}|.$ Hence  $\gG_{n}$ has diameter $n-1.$ This shows that the Stirling recursion provides the graph diameter automatically from initial values. 

\pf The first statement is due to Cayley, it follows from \,(\ref{MM})\, and induction. For the second claim we partition $B_{r}=X\cup Y_{1}\cup ...\cup Y_{n-1}$ where $X$ consists of the permutations which fix $n$ while for $0<j\leq n-1$ the set $Y_{j}$ consists of  the permutations in which $n$ succeeds $j$  in the same cycle. In other words, $X=B_{r}\cap \ins_{0}(G_{n-1})$ and $Y_{j}=B_{r}\cap \ins_{j}(G_{n-1}).$ Using the first part, as $\del$ decreases the number of cycles on $X$ we have $\del(X)=B_{r}(\gG_{n-1})$ and hence $|X|=|\del(X)|=|B_{r}(\gG_{n-1})|.$ Similarly, as $\del$ keeps  the number of cycles on $Y_{j}$ invariant  we have $\del(Y_{j})=B_{r-1}(\gG_{n-1})$ so that $|Y_{j}|=|\del(Y_{j})|=|B_{r-1}(\gG_{n-1})|$ for all $1\leq j\leq n-1.$ \dne

\subsection{The Ascent-Decent Pattern} \label{DS}

Our aim is to understand the behaviour of a path $P$ in $\gG_{n}$ when $\gG_{n}\hookrightarrow \gG_{n\!+\!1}$ is embedded via the   insertion map $\ins_{j}\!:\, G_{n}\to G_{n\!+\!1}$ for some $0\leq j\leq n.$  By Proposition~\ref{Embed22}\, we have that $P\subseteq \gG_{n}$ is a path in $\gG_{n}$ if and only if $\ins_{j}(P)\subseteq \gG_{n\!+\!1}$ is a path in $\gG_{n\!+\!1}.$  

Let $u\neq w$ be two vertices of $\gG_{n}$ and let \begin{align}\label{Path1}   P=P(u,w)=\big(u_{0},u_{1},...,u_{k-1},u_{k}\big)\,\,\,\text{with $u=u_{0}$ and $u_{k}=w$} \end{align}  be a path in $\gG_{n}$ from $u$ to $w$ of length $k.$ Then $P$ corresponds to a factorisation of $g=u^{-1}w$ as a product of $k$ transpositions $t_{i}\in G_{n}$ with $g=t_{1}t_{2}...t_{k}$ so that $u_{i}=u_{0}t_{1}...t_{i}$ for all $i=1..k.$ 

We say that $P$ has an {\it ascent at step }$i$ if ${\rm d}(e,u_{i})=  {\rm d}(e,u_{i-1})+1,$\,  or  a {\it descent at step} $i$  if ${\rm d}(e,u_{i})=  {\rm d}(e,u_{i-1})-1.$ Since the sign (as permutations) of the vertices along $P$ alternates it follows that each step is either an ascent or a descent. Accordingly, let $a,\,d$ be two letters and define the {\it ascent-decent pattern}\, of $P$ as $$\ad(P):=(\star_{1},...,\star_{k})\in \{a,\,d\}^{k}$$ where $\star_{i}=a$\, or $=d$ according as $P$ has an ascent or a descent at step $i.$ 

\medskip

\begin{prop}\label{ADThm}
Let $P$ be a path in $\gG_{n}$ and $j\in\{0..n\}.$ Then $\ad(\ins_{j}(P))=\ad(P).$
\end{prop}

\pf It is sufficient to consider a path $P=P(u, ut)$ of length $1,$ with transposition $t=(\alpha,\beta).$  Then $\ins_{j}(P)=P(\ins_{j}(u),\ins_{j}(ut)\,)=P(\ins_{j}(u),\ins_{j}(u)t),$ by Lemma~\ref{Ins}. We have $\ad(P)=(d)$ if and only if $\alpha$ and $\beta$ are in the same  cycle of $u.$ Similarly,  $\ad(\ins_{j}(P)\,)=(d)$ if and only if $\alpha$ and $\beta$ are in the same  cycle of $\ins_{j}(u).$  But the two conditions are the same: as $n\!+\!1\not\in \{\alpha,\beta\}$ it follows that  $\alpha$ and $\beta$ are in the same  cycle of $\ins_{j}(u)$ if and only if they are in the same  cycle of $u.$  \dne

To give an example  let  $u=(1,5)(2,4)(3)$ and $g\!=~(1,2)(1,3)(1,4)=(1,2,3,4)$ in  \sym{5}. This gives the path \begin{eqnarray}P&=&\big(u,\,u(1, 2),\,u(1,2)(1,3),\,u(1,2)(1,3)(1,4)\big)\nonumber\\
&=&\big((1,5)(2,4)(3),\,(1,5,2,4)(3),\,(1,5,2,4,3),\,(1,5,2)(4,3)\big)\nonumber
\end{eqnarray} in $\gG_{5},$ with ascent-descent pattern $\ad(P)=(a,a,d).$
If $j=1,$ for instance, put  $u_1:=\ins_1(u)=(1,6,5)(2,4)(3)\in \sym{6},$ obtained from $u$ by inserting $6$ after $1.$ This gives the path 
 \begin{eqnarray}P_{1}&=&\big(u_{1},\,u_{1}(1, 2),\,u_{1}(1,2)(1,3),\,u_{1}(1,2)(1,3)(1,4)\big)\nonumber\\
&=&\big((1,6,5)(2,4)(3),\, (1,6,5,2,4)(3),\,      (1,6,5,2,4,3),\,     (1,6,5,2)(4,3) \big)\nonumber
\end{eqnarray} in $\gG_{6}.$ It has the same ascent-descent pattern $\ad(P_{1})=(a,a,d)$ as $P.$
 
\bigskip
Conversely, given a  path $P$ in $\gG_{n\!+\!1},$ how can we obtain  a path $Q$ in $\gG_{n}$ with the same ascent-descent pattern as $P?$ By Proposition~\ref{ADThm}\,  it is sufficient  that $P$ is of the form $P=\ins_{j}(Q)$ for some $0\leq j\leq n.$ (If $j=0$ we may regard $P$ as a path in $\gG_{n}.)$ The question can be answered if $P$ is of the form $P(v,v_{p})=(v, v_{1},..,v_{p})$ where $v_{i}=vt_{1}...t_{i}$ for $i=1..p$ with transpositions $t_{1},..,\,t_{p}$ each fixing $n\!+\!1.$ In this case all vertices of $P$ are contained in the coset $vG_{n}$ and restricted to this set the deletion map  $\del\!:G_{n\!+\!1}\to G_{n}$ is a graph isomorphism. Its inverse is  $\ins_{j}\!:G_{n}\to vG_{n}$ where $j=v^{-1}(n).$  In particular, $P=\ins_{j}(\del (P)).$

Let $m$ and $k$ be non-negative integers. In the next proposition we consider  $S_m$ and $S_{k}$ as spheres in $\gG_{n\!+\!1}$ while  $S_{m-t}$ and $S_{k-t}$ are considered as spheres in $\gG_n.$ 

\bigskip

\begin{prop}[Deletion Lemma]
Let $g$ be a permutation in $G_{n\!+\!1}$ fixing $n\!+\!1.$ If $v\in G_{n\!+\!1}$ then 
\begin{align}
(v,vg) \in S_m \times S_k \text{~~~if and only if~~~} (\del(v),\del(vg)) \in S_{m-c} \times S_{k-c} \label{CLemma}
\end{align}
where $c=0$ if $v$ moves $n\!+\!1$ and $c=1$ otherwise. 
\end{prop}

\pf
By Theorem~\ref{BB1}\, we have  ${\rm d}_{\gG_{n\!+\!1}}(e,v)={\rm d}_{\gG_{n}}(e,\del(v))$ if $v$ moves $n\!+\!1$ and ${\rm d}_{\gG_{n\!+\!1}}(e,v)-1={\rm d}_{\gG_{n}}(e,\del(v))$ otherwise. Suppose that $g=t_1t_2..t_p$ is a shortest factorization of $g$ as a product of transpositions $t_{i}$ each fixing $n\!+\!1.$  Let $v_0:=v$ and $v_s:=v_{s-1}t_s$ for $s=1..p.$ Then  $P=P(v,vg)=\big(v_{0},..,v_{p}\big)$  is a path, and  it is clear that ${\rm d}_{\gG_{n\!+\!1}}(e,vg)$ is determined by ${\rm d}_{\gG_{n\!+\!1}}(e,v)$ and the entries in $\ad(P).$ By the argument above $Q:=\del(P)=P(\del(v),\del(vg))$ is a path with $\ad(P)=\ad(Q).$  Again, ${\rm d}_{\gG_{n\!+\!1}}(e,\del(vg))$ is determined by ${\rm d}_{\gG_{n\!+\!1}}(e,\del(v))$ and the entries in $\ad(Q).$  \dne

\subsection{Intersection Numbers for $\gG^{1}_{n}$}\label{sect32}

Let $g$ be a permutation of some set and suppose that $s$ denotes the support size of $g.$ If $n\geq s$  we may identify the set with a subset of $\{1..n\}$ and view $g$ as an element of \sym{n} fixing  $n-s$ elements of $\{1..n\}.$ Since $\Phi(\gG_{n}^{1};\,r,g)$ is a class function by Proposition~\ref{Cayley22} it is irrelevant in which particular way this is done. 
As before $\gG^{1}_{n}$ denotes the transposition Cayley graph. 

\medskip
\begin{thm} \label{Gamma1} Let $g$ be a permutation with support size $s\geq 0.$  Then there is a Stirling function $\big[{n\atop m}\big]_{g}$ with threshold $t=\max\{s,2\}$ so that $\Phi(\gG_{n}^{1};\,r,g)=\big[{n\atop n-r}\big]_{g}$ for all $n>t$ and all $r\geq 0.$\end{thm}

To determine $\Phi(\gG_{n};\,r,g)$ completely the initial values  $\Phi(\gG_{n};\,r,g)$ are required for $n=t,$  as discussed in Section 2.1.  For  small $t$ this is done easily by inspection.

\medskip
\pf Let $n>t.$ Then $n>s$ and we may assume that $g$ fixes $n.$ We partition $B_r\cap B_rg$ as $B_r\cap B_rg=X\cup Y_{1}\cup Y_{2}\cup...\cup Y_{n-1}$ where $X$ are the permutations fixing $n$ while $Y_{j}$ are the permutations which map $j$ to $n.$ In other words, $X=(B_r\cap B_rg)\cap \ins_{0}(G_{n-1})$ and $Y_{j}=(B_r\cap B_rg)\cap \ins_{j}(G_{n-1})$ for $0<j\leq n-1,$ compare to the proof of Theorem~\ref{BB1}. Note that $\del\!:G_{n}\to G_{n-1}$ is injective on each of these sets. 

Observe that $v$ belongs to $B_{r}\cap B_{r}g$ if and only if $v$ and $vg^{-1}$ belong to $B_{r}$ and so we may  apply the Deletion Lemma.  We have $v\in X$ if and only if $\del(v)\in B_r(\gG_{n-1},e) \cap B_r(\gG_{n-1},\del(g)).$ Hence $\del(X)=B_r(\gG_{n-1},e) \cap B_r(\gG_{n-1},\del(g)) $ and therefore $|X|=|\del(X)|=\Phi(\gG_{n-1};\,r,g)$ when $g$ is viewed as an element of $G_{n-1}.$

By the lemma we have $v\in Y_{j}$   if and only if $\del(v)\in B_{r-1}(\gG_{n-1},e) \cap B_{r-1}(\gG_{n-1},\del(g)).$ Hence $\del(Y_{j})=B_{r-1}(\gG_{n-1},e) \cap B_{r-1}(\gG_{n-1},\del(g))$ and therefore $|Y|=|\del(Y_{j})|= \Phi(\gG_{n-1},r-1,g)$ when $g$ is viewed as an element of $G_{n-1}.$ Taken together we have \begin{align}\label{XY}\Phi(\gG_{n};\,r,g)=\Phi(\gG_{n-1};\,r,g)+(n-1)\Phi(\gG_{n-1},r-1,g)\,\,.\end{align} In view of \,(\ref{STRNoo})\, and \,(\ref{STRNo})\, this is the required result. \dne


\bigskip
\section[title]{The Double-Transposition Cayley Graph}

Assume that $n\geq 5$ is an integer. Let $G'_{n}\subset G_{n}$ denote the alternating and symmetric groups  on $\{1..n\}$  respectively.  In this section $H$ denotes the set of all double-transpositions in $G_{n},$ these are the permutations of cycle shape $(a\,b)(c\,d).$ Then $H$ generates $G'_{n}$ and the resulting Cayley graph on $G'_{n}$ is the {\it double-transposition Cayley graph\,} $\gG_{n}^{2}:=\gG^{H}_{G'_{n}}.$ 
For $r\geq 0$ we set $S_r:=S_r(\gG_n,e)$ and 
 $B_r:=B_r(\gG_n,e).$ If $g\in G'_{n}$ we put $S_{r}(g):=S_r(\gG_n,g)$ and 
 $B_r(g):=B_r(\gG_n,g).$ We are interested in the intersection numbers 
$$\Phi(\gG_{n}^{2};\,r,g)=\abs{B_r \cap B_rg}$$ in $\gG_{n}$ for all $n\geq 5$ and  $r\geq 0.$  

\subsection{Metric Spheres in $\gG^{2}_{n}$}

We begin by determining  the spheres $S_{r}$ in $\gG_{n}^{2}$ for $r\geq 1.$ Recall that $|g|$ is the number of cycles of $g,$ including cycles of length $1.$ 

\medskip 
\begin{prop}\label{SG22}
Let $n\geq 5$ and let $g\in G'_{n}.$ Then $g$ belongs to $S_2\subseteq \gG_{n}^{2}$ if and only if either $\abs{g}= n-4$ or $g$ is a $3$-cycle. If $r\geq 3$ then $g$ belongs to $S_r\subseteq \gG_{n}^{2}$ if and only if $\abs{g} = n-2r.$ 
\end{prop}

\pf It suffices to analyze cycle types.  If $g=(1\ 2\ 3)$ then $g= (1\ 2)(4\ 5)\cdot (1\ 3)(4\ 5)$ belongs to $S_2$. Suppose that $\abs{g} = n-4$. If $g=(1\ 2\ 3\ 4)(5\ 6)$ then $g=(1\ 2)(3\ 4)\cdot (1\ 3)(5\ 6)$ belongs to $S_{2}.$ There are four other cycle types for  an element with $n-4$ cycles, and the following shows how each is expressed as a product of two double-transpositions:

\vspace{-5mm}
{\small
\begin{align*}
&(1\ 2)(3\ 4)\cdot(5\ 6)(7\ 8) = (1\ 2)(3\ 4)(5\ 6)(7\ 8),& & (1\ 2)(3\ 4)\cdot(1\ 3)(2\ 5) = (1\ 5\ 2\ 3\ 4),   \\
&(1\ 2)(3\ 4)\cdot (1\ 5)(3\ 6) = (1\ 2\ 5)(3\ 4\ 6),& &(1\ 2)(3\ 4)\cdot(1\ 5)(6\ 7) = (1\ 2\ 5)(3\ 4)(6\ 7).
\end{align*}}

\vspace{-5mm}
Hence  if $g$ is a $3$-cycle or $\abs{g} = n-4$ then $g$ belongs to $S_2$. Conversely, if  $g\in S_2$ then $g$ is a product of four transpositions and hence belongs to the ball of radius four in the transposition graph $\gG_n^{1}.$ Since $g$ is an even permutation we have $\abs{g} = n-2i$ with $i=0,1,2$ by Theorem~\ref{BB1}. As $H=S_1$ and $\{e\}=S_0$ we have proved the first part. 

For the general case, an element  $g$ in $B_r$ is a product of at most $2r$ transpositions and so has at least $n-2r$ cycles. Since $S_r= B_r\setminus B_{r-1}$ we have that for $r\geq 3$, if $g$ is in $S_r$, then $\abs{g}=n-2r$. Next assume that $g$ is in $S_3$. Then there exist $h$ in $H$ and $x$ in $S_2$ such that $g=xh$. Since multiplication by a transposition amounts to gluing together two disjoint cycles or splitting a cycle into two, see  \,(\ref{MM}), multiplication  by a double-transposition increases or decreases the number of cycles by two, or leaves this number unchanged. Since $x$ belongs to $S_2$ and $g$ belongs to $S_3$ we have $\nc{x}=n-4$. 
Hence $\nc{g}=n-6,\,n-4$ or $n-2$ and since $g$ is not in $B_2$ we have $\nc{g}=n-6$. We leave the proof, it is clear how the argument extends to elements in $S_r$ with $r>3.$ 
\dne

\subsection{The Stirling Recursion in the Double-Transposition Cayley Graph}\label{SRDT}

We  introduce some notation to enable us to link the graph $\gG_{n}^2$ to the transposition Cayley graph $\gG_{n}^{1}.$ We can not expect containment as graphs since the former has triangles while the latter does not. Nevertheless, the vertices of   $\gG_{n}^{2}$ are vertices of $\gG_{n}^{1}$ and edges in  $\gG_{n}^{2}$ correspond to certain paths of length $2$ in $\gG_{n}^{1}.$ This will allow us to use the Deletion Lemma also in this situation. In Theorem~\ref{Gamma2}\, we will show that also  here $\Phi(\gG_{n}^{2};\,r,g)$ satisfies the Stirling relation.

\bigskip
Let $S^{1}_r:=S_r(\gG^{1}_{n},e)$ and $B^{1}_r:=B_r(\gG_{n}^{1},e)$ be the sphere and  ball in the transposition Cayley graph $\gG_n^{1}$ on \sym{n}.
We emphasize,  the distance functions are different, one must think of the sets defined in each case. Note that $g\in G_{n}$ is even if and only if $g\in S^{1}_{r}$ for sufficiently large even $r.$ 

Let $r\geq 0$ and  define $Z_r\subset  \sym{n}$  by
\begin{align}
Z_r:= 
\begin{cases}
\,\,B^{1}_r \cap \alt{n} \hspace{2.5cm}\text{\ if\ } r \text{\ is even, and} \medskip\\
\,\,B^{1}_r \cap (\sym{n}\setminus \alt{n})\hspace{.8cm} \text{\ if\ } r \text{\ is odd}\,,
\end{cases} \label{ZR}
\end{align}
or equivalently, 
\begin{align}
Z_r:= 
\begin{cases}
\,\,S^{1}_0\cup S^{1}_2\cup...\cup S^{1}_{r-2}\cup S^{1}_r \text{\quad if\ } r \text{\ is even, and} \medskip\\
\,\,S^{1}_1\cup S^{1}_3\cup... \cup S^{1}_{r-2}\cup S^{1}_r \text{\quad if\ } r \text{\ is odd}.
\end{cases} \label{ZR1}
\end{align}
For $g$ in $G_n$ and $r\geq 0$ let 
\begin{align}
I_g(n,r) := \abs{B^{1}_r \cap Z_rg}.\label{DIN221}
\end{align}
Next we show that the numbers $I_g(n,r)$ satisfy a familiar recursion. The following proposition is the key to the remainder of the paper. 

\medskip
\begin{prop}\label{INRPropG22}
Let $n\geq 2$ and let $g\in {\rm Sym}(n)$ have support size $s<n.$ Then for all  $r\geq 0$ we have   
\begin{align}
I_g(n,r) = I_g(n-1,r) + (n-1)I_g(n-1,r-1)\,.\label{INRG22}
\end{align}
In particular, there is a Stirling function $[{n\atop m}]_{g}$ with threshold $t=\max\{2,s\}$ so that $[{n\atop n-r}]_{g}=I_{g}(n,r)$ 
for all $n>t$ and all $r\geq 0.$ 
\end{prop}



\pf  We regard $g$ as an element of \sym{n} fixing $n.$ As in the proof of Theorem~\ref{BB1} \,we partition $B^{1}_r \cap Z_rg$ into  $B^{1}_r \cap Z_rg=X\cup Y_{1}\cup Y_{2}\cup...\cup Y_{n-1}$ where $X=(B^{1}_r \cap Z_rg)\cap \ins_{0}(G_{n-1)}$ and $Y_{j}=(B^{1}_r \cap Z_rg)\cap\ins_{j}(G_{n-1})$ for $1\leq j\leq n-1.$
Using the  Deletion Lemma and the same arguments as in Theorem~\ref{Gamma1}\, we have $
\abs{X} = I_g(n-1,r)$ and $\abs{Y_{j}} = I_g(n-1,r-1).$ Hence $I_{g}(n,r)$ satisfies \,(\ref{INRG22}). Using  \,(\ref{STRNo})\, and \,(\ref{STRNoo}) one may define a  Stirling function $[{n\atop m}]_{g}$ with $[{n\atop n-r}]_{g}=I_{g}(n,r).$   \dne

We are ready to determine the intersection numbers in the double-transposition graph $\gG^{2}_{n}.$ If $r\geq 2$ we have $B_{r}=B^{1}_{2r}=Z_{2r}$ by Proposition~\ref{SG22} and the definition of $Z_{r}.$ Hence 

\vspace{-3mm}
\begin{align}\label{XaX}
I_g(n,2r) = \abs{B^{1}_{2r} \cap B_{r}g} = \abs{B_r \cap B_rg}=\Phi(\gG_{n},r,k)\,.
\end{align}
(Note that $I_g(n,2) > \abs{B_1 \cap B_1g} $ since $B_1$ does not contain $3$-cycles.) Now apply Proposition~\ref{INRPropG22}\, to obtain the following theorem. As before, if $g$ is a permutation with support size $s$ we view it as an element of \sym{n} for all $n\geq s.$  

\medskip
\begin{thm} \label{Gamma2} Let $g$ be an even permutation with support size denoted by $s\geq 0.$ Then there exists a Stirling function $\big[{n\atop m}\big]_{g}$ with threshold $t=\max\{s,4\}$ so that $\Phi(\gG_{n}^{2};\,r,g)=\big[{n\atop n-2r}\big]_{g}$ for all $n>t$ and all $r\geq 2.$\end{thm}

The earlier comments apply:  In order to determine $\Phi(\gG^{2}_{n};\,r,g)$ explicitly  one has  to compute $\Phi(\gG_{t}^{2};\,r,g)$ for $0<r<t$ directly. The theorem does not cover  $\Phi(\gG^{2}_{n};\,1,g),$ this can be determined directly from Proposition~\ref{SG22}. It turns out that $\Phi(\gG^{2}_{n};\,1,g)$ is independent of $n$ if $g\in S_{2}$ and quadratic in $n$ if $g\in S_{1}.$ Evidently $\Phi(\gG^{2}_{n};\,1,g)=0$ if $g\in S_{r}$ for $r\geq 3.$

\bigskip


\section {The $k$-Transposition Cayley Graph $\gG^{k}_{n}$}

Let $k\geq 3$ and $n$ be integers with $n\geq 2k.$ In $G_{n}=\sym{n}$ let $H$ be the conjugacy class of all $k$-transpositions, that is,  all permutations of cycle shape $(a_{1}\,b_{1})(a_{2}\,b_{2})...(a_{k}\,b_{k}).$   If $k$ is even then $H$ generates the alternating group $G'_{n}=\alt{n},$ and in this case  $\gG^{k}_{n}$ denotes the Cayley graph $\gG^{H}_{G'_{n}}$ with  vertex set $G'_{n}$. If $k$ is odd then $H$ generates the symmetric group $G_{n},$ and in this case  $\gG^{k}_{n}$ denotes the Cayley graph $\gG^{H}_{G_{n}}$ with vertex set $G_{n}$. This graph is the  {\it $k$-transposition Cayley graph\,} on the symmetric or alternating group of degree $n.$  If the context is clear we write $\gG_{n}=\gG^{k}_{n}.$

As before  $B_r=B_r(\gG^{k}_n,e)$ and $S_r=S_r(\gG^{k}_n,e)$ denote the ball and sphere in $\gG^{k}_{n}.$ In order to compare $\gG^{k}_{n}$ to $\gG^{1}_{n}$  we  use $B^{1}_r$ and $S^{1}_r$ to refer to the ball and the sphere of radius $r$ in the transposition Cayley graph $\gG_n^1.$ Again, the distance functions are different and one should think of the sets defined in each case.

\subsection{Metric Spheres  in  $\gG^{k}_{n}$}

As we have seen,  in  $\gG_n^{1}$ the distance from  $e$ to $u$ in $G_{n}$ is  $n-\abs{u}$ where $\abs{u}$ is the number of cycles in the cycle decomposition of $u$. This distance is invariant under the embedding $\ins_{0}\!: \gG^{1}_{n}\to \gG^{1}_{n+1},$ and the same is  true for $\ins_{0}\!: \gG^{2}_{n}\to \gG^{2}_{n+1}$ if $n\geq 5.$ 
However, the situation becomes  different for $k-$transposition Cayley graphs when $k$ is greater than two. 

For instance,  $(1\ 2\ 3\ 4\ 5)$ belongs to $S_4(\gG_6^3,e)$ while it belongs to $S_2(\gG_n^3,e)$ for all $n\geq 7.$ Therefore the embedding $\ins_{0}\!:\gG_{n}^{k}\to\gG^{k}_{n+1}$ is not isometric in these cases. These are expected irregularities due to the closeness of $2k$ to $n.$ For $n$ sufficiently large in relation to $2k$ we expect the embedding to be isometric. This is confirmed in Propositions~\ref{ALLSPK}\, and~~\ref{SG2KE}\, below. We begin with a lemma.


\medskip
\begin{lem}\label{S2Graph} Let $k\geq 1$ be an integer and let $g$ is an element of \sym{n}.  Suppose that $n=2k+|g|$ where $|g|$ is the number of cycles in $g.$ Then $g$ is the product of two $k$-transpositions on $\{1..n\}.$  
\end{lem}

\pf We assume that $g$ is written as 
\[g = a_1a_2a_3\cdots a_{2t-1}a_{2t}d_1d_2\cdots d_{j-1}d_j
\] 
where
$$
a_i =\ (\alpha_{i1}\ \alpha_{i2} \ldots \alpha_{i,2m_i})\text{\ and\ }
d_i =\ (\beta_{i1}\ \beta_{i2}\ldots \beta_{i,2q_i+1})
$$
are cycles of even and odd length respectively.  The number of even cycles must be even since $\abs{g}=n-2k$. Here we have 
\[
2k = \sum_{i=1}^{2t}(2m_i-1) + \sum_{i=1}^j 2q_i
\]
and therefore 
\begin{align}
k=\sum_{i=1}^{2t}m_i + \sum_{i=1}^j q_i -t.  \label{NOCy}
\end{align}
For each $i$ we may express $a_{i}$ and $d_{i}$ as $a_i= b_i\cdot c_i$ and $d_i=e_i\cdot f_i$ 
where
\begin{align*}
b_i =\ &\  (\alpha_{i1}\ \alpha_{i,2m_i-1})(\alpha_{i2}\ \alpha_{i,2m_i-2})\ldots (\alpha_{i,m_i-1}\ \alpha_{i,m_i+1}),\\
c_i =\ &\  (\alpha_{i1}\ \alpha_{i,2m_i})(\alpha_{i2}\ \alpha_{i,2m_i-1})\ldots (\alpha_{i,m_i}\ \alpha_{i,m_i+1}),\\
e_i =\ &\  (\beta_{i1}\ \beta_{i,2q_i})(\beta_{i2}\ \beta_{i,2q_i-1})\ldots (\beta_{i,q_i}\ \beta_{i,q_i+1}),\\
f_i =\ &\  (\beta_{i1}\ \beta_{i,2q_i+1})(\beta_{i2}\ \beta_{i,2q_i})\ldots (\beta_{i,q_i}\ \beta_{i,q_i+2})
\end{align*}
in terms of disjoint cycles. If we let
\begin{align*}
x = \ &\  b_1b_3\ldots b_{2t-1}c_2c_4\ldots c_{2t}e_1e_2\ldots e_j \text{\ \ \ and \ }\\
y = \ &\  b_2b_4\ldots b_{2t}c_1c_3\ldots c_{2t-1}f_1f_2\dots f_j 
\end{align*}
then from (\ref{NOCy}) we have
\begin{align*}
\abs{x}= \ &\ (m_1-1) + (m_3-1) +\cdots + (m_{2t}-1) \\
                  \ &\ + m_2+m_4+\cdots+m_{2t}+q_1+\cdots+q_j\\
                = \ &\ k.
\end{align*}
Also $\abs{y}=k,$ and so both $x$ and $y$ are $k$-transpositions with $g= x\cdot y.$
\dne

\medskip
\begin{prop}\label{SG2KO}
For $k\geq 3$ and  $n\geq 4k$ let $S_{2}$ denote the sphere of radius $2$ in $\gG_n^{k}.$   Then $S_2 = \{\,g\in \sym{n} : \nc{g}=n-2t \text{\, for some\,\,\,} t=1,2\ldots,k\,\}$. \end{prop}

\begin{proof}
Let $H\subset \sym{n}$ be the set of all $k$-transpositions.  Any product $g$ of two elements in $H$ has at least $n-2k$ cycles. Also, $\abs{g}$ and $n$  have the same  parity. The case $\abs{g} = n-2k$ is done by Lemma~\ref{S2Graph}. Let $g$ be an element with $\abs{g}=n-2t$ for some $t < k,$ and suppose that $k=t+p$ for some $p$. By Lemma~\ref{S2Graph}\, we have that $g=g_1g_2$ for some $g_1,g_2$ of cycle type $1^{n-2t}2^t$. Since $n\geq 4k$ there exist disjoint transpositions $t_1,t_2,\ldots,t_p$ such that $Supp(t_i)$ does not intersect $Supp(g_1)\cup Supp(g_2)$ for all $i=1,\ldots, p$. 
As $g_1$ and $g_2$ are of cycle type $1^{n-2t}2^t$ and $k=t+p$, we have that
\[
g = (g_1t_1t_2\ldots t_p) \cdot (g_2t_1t_2\ldots t_p) 
\]
is a product of $k$-transpositions. 
The proof is complete. 
\end{proof}

\subsection{The Case of odd $k>2$}

We can now extend the characterisation of the spheres in  $\gG_n^k$ to $k\geq 3.$ First suppose that $k$ is odd. 

\medskip
\begin{prop}\label{ALLSPK}
For the odd integer $k\geq 3$ and $n\geq 4k$ let $S_t$ denote the sphere of radius $t$  in $\gG_n^k.$  Then

\vspace{-10mm}
\begin{align*}
S_0 =\ &\{\,e\,\},\ \ S_1 = H, \text{\,the set of all $k$-transpositions,}\\
S_2 =\ &\{\,g\in G_{n}\,\colon \nc{g}\equiv n (\text{mod\ }2) \text{\ and } n-2\geq \nc{g}\geq n-2k \,\},\\
S_3 =\ &\{\,g\in G_{n}\,\colon g \not\in H,\ \nc{g}\not\equiv n\ (\text{mod\ }2) \text{\ and } n-1\geq \nc{g}\geq n-3k \},\\
S_{2(r+1)} =\ &\{\,g\in G_{n}\,\colon \nc{g}\equiv n\ (\text{mod\ }2) \text{\ and } n-2rk> \nc{g}\geq n-2(r+1)k \},\\
S_{2r+3} =\ &\{\,g\in G_{n}\,\colon \nc{g}\not\equiv n\ (\text{mod\ }2) \text{\ and } n-(2r+1)k> \nc{g}\geq n-(2r+3)k \},
\end{align*}
for all $r\geq 1.$
\end{prop}

\pf 
The first three statements are clear from  Proposition \ref{SG2KO}. If $g\in B_r$ then $g$ is a product of at most $kr$ transpositions so that $\abs{g}\geq n-rk.$ (Consider $g$ as an element of a suitable ball in $\gG^{1}_{n}.)$ This gives the lower bound for the cycle numbers in each $S_{r}.$
Since $H$ is a set of odd permutations, for  $g$ in $B_r$ the parities of $\abs{g}$ and $n$ are the same if and only if $r$ is even. Further, 
\[
(1\ 2)= (1\ 2)(3\ 4)(5\ 6)\cdot(1\ 2)(3\ 5)(4\ 6)\cdot(1\ 2)(4\ 5)(3\ 6)
\]
and so  $S_3$ contains all transpositions when $k=3.$  For odd $k>3$ similarly consider products of the  kind $\big((1\ 2)(3\ 4)(5\ 6)\cdot a\big)\cdot \big((1\ 2)(3\ 5)(4\ 6)\cdot b\big)\cdot\big((1\ 2)(4\ 5)(3\ 6)\cdot b\,a \big)$ where

\vspace{-10mm}
\begin{align*}
a=\ & (a_{4}\ b_{4})(a_{5}\ b_{5})...(a_{k-1}\ b_{k-1})( a_{k-1}\ b_{k}) \text{\,\,\,and \,\,}
\\
b=\ & (a_{4}\ a_{5})(b_{4}\ b_{5})...(a_{k-1}\ a_{k})(b_{k-1}\ b_{k})
\end{align*}
with $a_{i},\,b_{i}\not\in\{1..6\}.$ Therefore $S_3$ contains all transpositions for all odd $k\geq 3.$

\medskip
Next, let $g$ be an odd permutation not in $H$ with $n-3\geq\abs{g}\geq n-3k$. 
We claim that $g$ is in $S_3$. Clearly $\abs{g}\not \equiv n$ (mod~2). Suppose first that  $\abs{g}\leq n-2k-1$. Then there exist $g_1$ in $S_2$ and $t\leq k$ such that $g=g_1x$ with $\abs{x}= n-t$ and $\abs{x}\not\equiv n$ (mod 2). Since $g_1$ belongs to $S_2$ we have  $g=h_1h_2x$ with $h_1,h_2$ in $H$. Moreover, $\abs{h_2x}= n-(k+t)\leq n-2k$ and $\abs{h_2x} \equiv n$ (mod 2). Hence $h_2x$ is in $S_2$ and therefore $h_2x=h_3h_4$ for some $h_3,h_4$ in $H$. Then  $g=h_1h_2x=h_1h_3h_4$ is in~$S_3$. 

\bigskip
On the other hand, if $n-3\geq \abs{g}\geq  n-2k+1$ then $g$ has a cycle of length at least two in its cycle decomposition, say $(\alpha_1 \ldots \alpha_s)$. Therefore $g=g_1(\alpha_1 \ \alpha_s)$ for some $g_1$ with $n-2\geq \nc{g_1}\geq n-2k+2$. This shows that $g_1$ belongs to $S_2$ and hence $g_1 =h_1h_2$ for some $h_1,h_2$ in $H$. It follows that $g = h_1h_2(\alpha_1 \ \alpha_s)$ and $n-2\leq \abs{h_2(\alpha_1 \ \alpha_s)}\leq n-2k$. Hence $h_2(\alpha_1 \ \alpha_s) = h_3h_4$ for some $h_3,h_4$ in $H$. Therefore $g=h_1h_2(\alpha_1 \ \alpha_s)=h_1h_3h_4$ and $g$ belongs to $B_3.$ Clearly, $g$ does not belong to $B_2$. So we have 
\[
S_3 = \{g\colon g \not\in H,\ \nc{g}\not\equiv n (\text{\ mod\ }2) \text{\ and } n-1\geq \nc{g}\geq n-3k \}.
\]

For the case of  $S_{2(r+1)}$ with $r\geq 1$, it is easy to see that any permutation in $S_{2(r+1)}$ has at least $n-2(r+1)k$ cycles. It then remains to show that 
\[
A(r+1):= \{g\ \colon \nc{g}\equiv n \ (\text{\ mod\ }2) \text{\ and } n-2rk> \nc{g}\geq n-2(r+1)k \} \subseteq S_{2(r+1)}. 
\]
We prove this by the induction on $r\geq 0,$ for $r=0$ this is done. Let $g$ be in $A(r+1)$. Then $\abs{g}= n-2rk-2t$ for some $t=1,\ldots,k$.  
By hypothesis we have $g=g_1x$ where $\abs{x} = n-2t$ and $g_1\in A(r)$. Hence $g_1$ belongs to $S_{2r}$ and therefore $g=h_1h_2\cdots h_{2r}x$. Also, we have  $n-1\geq \abs{h_{2r}x}\geq n-3k$ which implies that $h_{2r}x = h'_1h'_2$ for some $h_1',h_2'$ in $H=(2^k)$. Hence, $g$ belongs to $S_{2(r+1)}$. 

The case $S_{2r+3}$ can be handled similarly. Note that the parities of $\abs{g}$ and $n$ agree automatically. \dne


\begin{cor}\label{CALLSPK}
For odd  $k\geq 1$ and $n\geq 4k$ the diameter of $\gG_n^k$ is  $\max\left(2\left\lceil \frac{n-2}{2k} \right\rceil, 2\left\lceil \frac{n-k-1}{2k} \right\rceil +1\right)$ when $n$ is even and $\max\left(2\left\lceil \frac{n-k-2}{2k} \right\rceil + 1, 2\left\lceil \frac{n-1}{2k} \right\rceil\right)$ when $n$ is odd.
\end{cor}


\pf It is sufficient to verify that the value $r$ given for the diameter satisfies $S_{r}\neq\emptyset=S_{r+1}.$ \dne 

\subsection{The Case of even $k>2$}

It remains to determine  the spheres in  $\gG_n^k$ when  $k$ is even. Here the situation is even more regular:

\medskip
\begin{prop}\label{SG2KE}
For the even integer $k\geq 2$ and $n\geq 4k$ let $S_{r}$ denote the sphere of radius $r$ in $\gG_n^{k}$ and let $G'_{n}={\rm Alt}(n).$ Then
\begin{align*}
S_0 =\ &\{\,e\,\},\ \ S_1 = H, \text{\,the set of all $k$-transpositions,}\\
S_2 =\ &\{\,g \in G'_{n}\colon g \not\in  H \text{\ and } n-2\geq \nc{g}\geq n-2k \,\},\\
S_{r} =\ &\{\,g\in G'_{n} \colon n-(r-1)k> \nc{g}\geq n-rk \,\}.
\end{align*}
for all $r\geq 3.$
\end{prop}

\pf
It is clear that $S_{r} \subseteq B_r \subseteq \{g \in \alt{n}\ \colon \nc{g}\geq n-rk \}$ for all $r\geq 3.$ It suffices to show for all $r\geq 2$ that an even permutation $g$ with $\abs{g}\geq  n-rk$ belongs to $B_r.$ 
We prove this by induction on $r\geq 2;$ the case  $r=2$ follows from Proposition~\ref{SG2KO}. Suppose that $r\geq 3$.  Let $g$ be a permutation with $\abs{g}=n-(r-1)k-m$ where $2\leq m \leq k$ and $m\equiv 0 $ (\,mod 2\,). Hence there must be permutations $g_1$ and $x$ such that $g=g_1x$ with $\abs{g_1}=n-(r-1)k$ and $\abs{x}=n-m.$  By the induction hypothesis $g_1$ belongs to $B_{r-1}.$ Hence $g=h_1h_2\ldots h_{i-2}h_{r-1}$ for some $h_1,h_2,\ldots, h_{r-1}$ in $H.$ Since $\abs{h_{r-1}x} \geq n-2k,$  by the induction hypothesis, there exist $h'_1,h'_2$ such that $h_{r-1}x=h'_1h'_2$. Therefore $g=g_1x=h_1h_2\ldots h_{r-2}h_{r-1}x =h_1h_2\ldots h_{r-2}h'_1h'_2,$ and  so $g$ belongs to $B_r.$ \dne


\begin{cor}\label{CSG2KE}
For $k\geq 2$ even and $n\geq 4k$ the diameter of $\gG_n^k$ is $\left\lceil \frac{n-2}{k} \right\rceil.$
\end{cor}

\pf As before, verify that  the value $r$ given for the diameter satisfies $S_{r}\neq\emptyset=S_{r+1}.$ \dne

\section{Metric Intersection Numbers  in $\gG^{k}_{n}$} 

In this section we prove the main results  on  the intersection numbers in the $k$-transposition Cayley graph $\gG^{k}_{n}$ stated in the introduction. 
Let $k\geq 3$ and $n\geq 4k.$ Here we determine the intersection numbers for the $k$-transposition Cayley graph $\gG^{k}_{n}.$ As before we need to analyze certain sets in terms of the transposition Cayley graph $\gG^{1}_{n}.$ In particular, $B^{1}_r$ denotes the ball in $\gG^{1}_{n}$ of radius $r$ with the identity element as centre. 

Let $g$ be a permutation of some set  with  support size ${\rm supp}(g)=s.$  As $\Phi(\gG_{n}^{k};\,r,g)$ is a class function we may view $g$ as an element of \sym{n} for any $n\geq s.$ In (\ref{ZR}) and (\ref{ZR1}) we defined \begin{align}
Z_r:= 
\begin{cases}
\,\,B^{1}_r \cap \alt{n} \hspace{2.5cm}\text{\ if\ } r \text{\ is even, and} \medskip\\
\,\,B^{1}_r \cap (\sym{n}\setminus \alt{n})\hspace{.8cm} \text{\ if\ } r \text{\ is odd,}
\end{cases} 
\end{align}
giving the partition  $B_{r}^{1}=Z_{r}\,\dot{\cup}\, Z_{r-1},$ the two sets being the union of all spheres  contained in $B_{r}^{1}$ of even, respectively odd, permutations depending on the parity of $r.$ This gives rise to 
\begin{align}
\big[{n \atop n-r}\big]_{g}:=|B^{1}_{r}\cap Z_{r}g|\,.\,\label{45}
\end{align}
Proposition~\ref{INRPropG22}\, showed that $\big[{n \atop m}\big]_{g}$ is a Stirling function with threshold $s,$ and  by  Theorem~\ref{Gamma2}\, this function determines the metric intersection numbers in $\gG_n^{2}.$ First we consider the case when $k$ is odd. 

\medskip
\begin{thm} \label{Gamma7} Let $k\geq 3$ be  an odd integer and let $g$ be a permutation  with support size $s.$ Denote by $\gG^{k}_{n}$  the $k$-transposition Cayley graph on ${\rm Sym}(n).$ Then $$\Phi(\gG_{n}^{k};\,r,g)=\big[{n\atop n-rk}\big]_{g} \,\,+\,\, \big[{n\atop n-(r-1)k}\big]_{g}$$ for all $n> t:=\max\{s,4k\}$ and all $r\geq 3$ where $\big[{n\atop m}\big]_{g}$ is the Stirling function in \,(\ref{45}).
\end{thm}

\pf Let $B_{r}=B^{k}_r\,$ be the ball of radius $r$ in $\gG_n=\gG^{k}_{n}.$ We partition $B_{r}=X\cup Y$ where $X=B_{r}\cap \alt{n}$ and  $Y=B_{r}\setminus X$ are the even and odd permutations in $B_{r}$ respectively.  Assume that $n\geq \max\{s,4k\}.$ 

Suppose that $g$ is an even permutation. Then by considering parities  
$X\cap Yg=Y\cap Xg=\emptyset$ and so
\begin{align}
B_r\cap B_rg = (X\cap Xg)~\dot{\cup}~(Y\cap Yg). \label{B2K}
\end{align}

If $r$ is even  then by Proposition \ref{ALLSPK}  
\begin{align}
X = \bigcup_{i=0\,..\,\frac{rk}{2}} \,\,\,\,S^{1}_{2i} \label{E2k11}
\end{align}
and
\begin{align}
Y  = \bigcup_{i=1\,..\,\frac{(r-1)k+1}{2}}\,\,\,\, S^{1}_{2i-1} \,\, . \label{4E2k11}
\end{align}
Similarly, if $r$ is odd then
\begin{align}
X = \bigcup_{i=0\,..\,\frac{(r-1)k}{2}}\,\,\,\,S^{1}_{2i} 
\end{align}
and
\begin{align}
Y =  \bigcup_{i=1\,..\,\frac{rk+1}{2}}\,\,\,\,S^{1}_{2i-1}\,.\label{E2k12}
\end{align}
\noindent 
(Note that \,(\ref{4E2k11})\, and \,(\ref{E2k12})\, do not hold for $r\leq 2.)$  Therefore we have 
\[
\Big\vert\,(X\cap Xg) ~\dot{\cup}~ (Y\cap Yg) \,\Big\vert =  \qfrac{n}{n-rk}_g +\qfrac{n}{n-(r-1)k}_g
\]
from the definition of $\qfrac{n}{m}_g$ and  \,(\ref{E2k11})--(\ref{E2k12}). Hence as required 
\begin{align}\Big\vert\, B_r\cap B_rg\, \Big\vert = \qfrac{n}{n-rk}_g +\qfrac{n}{n-(r-1)k}_g\label{Paola}
\end{align}
from \,(\ref{B2K}). 

\bigskip
If $g$ is an odd permutation then   $X\cap Xg = Y\cap Yg =\emptyset$ by the parity argument and hence $
B_r\cap B_rg  = (X\cap Yg)~\dot{\cup}~ (Y\cap Xg).$ Applying similar arguments shows that \,(\ref{Paola})\, holds also in this case.  \dne

Next suppose that $k$ is even and $n\geq 5.$ In this case the $k$-transpositions generate \alt{n}.  By  Proposition~\ref{SG2KE}\, the metric structure of $\gG^{k}_{n}$ is quite simple when $r\geq 3$ and $n\geq 4k.$ We have $B_{r}(\gG^{k}_{n},e)=\{g\in\alt{n}\,:\, |g|\geq n-rk\}$ and by Proposition~\ref{SG22} this is equal to $B_{\frac12 rk}(\gG^{2}_{n},e).$ As a consequence we have

\medskip
\begin{thm} \label{Gamma8} Let $k\geq 2$ be  an even integer and let $g$ be an even permutation  with support size $s\geq 0.$ Denote by $\gG^{k}_{n}$  the $k$-transposition Cayley graph on ${\rm Sym}(n).$ Then $$\Phi(\gG_{n}^{k};\,r,g)=\Phi(\gG_{n}^{2};\,\frac12 rk,g)$$ for all $r\geq 3$ and $n\geq 4k.$ In particular, there is a  
Stirling function $\big[{n\atop m}\big]_{g}$ with threshold $t:=\max\{s,4k\}$ so that   
$$\Phi(\gG_{n}^{k};\,r,g)=\big[{n\atop n-rk}\big]_{g}$$ for all $n> t$ and $r\geq 3.$
\end{thm}

This concludes the proof of Theorems~\ref{main1} and~\ref{main2}.

\section{Applications  and Open Problems} 

We have concentrated on the metric structure of the $k$-transposition Cayley graph and generalized Stirling functions. Here we comment on the connection to error correction and reconstruction problems. 
Much the same questions arise for other Cayley graphs on the symmetric group, and these are mentioned below.  

\subsection{Error Graphs and the Maximum  of $\Phi(\gG_{n}^{k};\,r,g)$}

Error graphs are used to describe reconstruction or correction procedures when 
data is subjected to transformations or corrupted by errors. We refer to the original papers~\cite{lv1,lv2,lv3} by Levenshtein
and to~\cite{js} where we began to analyzed the ordinary transposition Cayley graph from the viewpoint of error graphs and reconstruction in groups. On this topic we recommend also a recent survey~\cite{Kon} by Elena Konstantinova which provides a wealth of information in situations which are particularly relevant for applications in computing and bio-informatics.   

For the purpose of reconstruction (or error correction) the $k$-transposition graph  arises in the following situation. We are interested in $n$-strings $x=(x_{1},...,x_{n})$ of  symbols $x_{i}$  from some alphabet which are subjected to transformations (or errors) $h:x\mapsto x'$ which involve a transposition of the coordinates of $x.$ More specifically, the coordinate transposition is of the  form $x_{i_{1}}\leftrightarrow x_{j_{1}},.., x_{i_{k}}\leftrightarrow x_{j_{k}}$ for some fixed $k\geq 1$ with  $i_{1},j_{1},..,i_{k},j_{k}\in \{1..n\}$ pairwise distinct. We write $x'=x_{h}$ to indicate that $x'$ is obtained by applying the permutation $h\in\sym{n}$ to the coordinate positions in  $x.$ Evidently $h$ now represents a $k$-transposition in the sense defined in this paper.  More generally, we may fix some $r\geq  1$ and allow up to $r$ such transpositions to occur independently, resulting in a distortion $x'=x_{h_{1}...h_{r'}}$ of $x$ with $r'\leq r.$ 

Transpositional transformations  of this kind on sequences  account for several of the basic genome rearrangements that underlie gene evolution, see for example the book~\cite{Fer}\, by Fertin {\it et al} on the combinatorics of genome rearrangements and Section 3 of Festa's article~\cite{PF}. The biologically relevant transformations will be restricted further, meaning that only certain kinds of $k$-transpositions occur in the process.  Nevertheless, the basic process remains the same. 

If we are given  a set $x^{1},..,x^{N+1}$ of distorted versions of $x,$ allowing for up to $r$ transpositional errors given by $k$-transpositions to have happened independently, is this information sufficient to determine $x$ uniquely? It is clear that this is the case for all $x$ if and only if $N\geq \Phi(\gG^{k}_{n};\,r,g)$ for all $g\in \sym{n}.$ Reconstructability is therefore determined by 
$$ N(\gG^{k}_{n},r):=\max\{ \Phi(\gG^{k}_{n};\,r,g)\,:\,e\neq g\in\sym{n}\,\}.$$
This reconstruction number has been determined for $\gG^{1}_{n}$ in~\cite{js}. It is shown that  for sufficiently large $n$  the maximum occurs uniformly when $g$ is  a $3$-cycle, that is $N(\gG^{1}_{n},r)=\Phi(\gG^{1}_{n};\,r,(1\,2\,3))$ for all $r.$
 The same remains true for $k$-transposition Cayley graphs with arbitrary $k\geq 1,$ this result will appear in~\cite{Teera}. The reconstruction numbers for $k$-transposition Cayley graphs are therefore  governed by the Stirling recurrence for all $k\geq 1.$

\subsection{Products of Conjugacy Classes}

We return to the description of Cayley graphs in Section~\ref{sect22}. Let $G$ be a group and $H$ a set of generators satisfying the three conditions stated at the beginning of this section. It is convenient to represent a set $X$ of group elements as the element $\sum_{g\in X}\,g$ in the group ring $\cC $ of $G$ over $\cC G,$  and allowing for a small abuse of notation, we denote this sum also by $X.$ 

If $X\cdot X$ denotes the product in the group ring then $g$ appears with non-zero coefficient in  $X\cdot X$ if and only if $g$ is an element of the set $XX=\{xx'\,:\,x,x'\in X\}.$ If we apply this to the set $H$ then it becomes an element  of the centre of  $\cC G$ since $H$ is invariant under conjugation. Furthermore, metric spheres and balls can be represented by product in  $\cC G$ since  $$B_{r}(\gG^{H}_{G})=\big\{ g\in G\,:\,\text{$g$ has non-zero coefficient in $(e+H)^{r}$}\big\}.$$
Therefore questions about the distance statistics in $k$-transposition graphs could be rephrased entirely in terms of the multiplication in the centre of the group algebra.  There is an extensive literature on this topic;\, as a route into it we suggest the Springer Lecture Notes~\cite{arad1} by Arad and Herzog. We also mention Goupil's papers~\cite{ag1,ag2} and Diaconis' treaty~\cite{dia}. While the multiplication constants in the centre of  $\cC G$ can be worked out in principle from the characters of \sym{n},  see  James and Liebeck~\cite{gj}, this turns out to be ineffective in practice. The multiplication constants are known explicitly only for few types of  conjugacy classes, see Goupil~\cite{ag1}. The Propositions~\ref{SG2KO},~\ref{ALLSPK}~and~\ref{SG2KE} contain some information about this problem and it is feasible to extend this to produce explicit formulae for the multiplication constants when one factor is a class of $k$-transpositions. 

It appears to be a very natural problem to determine  the metric intersection numbers in symmetric groups when $H$ is some other conjugacy class. For instance, if $H$ is the set of all $3$-cycles and $G=\alt{n},$ what will replace the Stirling recursion  controlling the behaviour of  $\Phi(\gG^{H}_{G};\,r,g)?$


\end{document}